\documentclass[11pt,reqno]{amsart}

\usepackage[utf8]{inputenc}
\usepackage[T1]{fontenc}
\usepackage{lmodern}
\usepackage{amsmath,amssymb,amsthm,amscd,mathtools}
\usepackage[mathscr]{euscript}
\usepackage{mathrsfs}
\usepackage{enumitem}
\usepackage{xcolor}
\usepackage[colorlinks=true,linkcolor=blue,citecolor=black,urlcolor=black]{hyperref}
\hypersetup{pdftitle={A Separable Banach Space with a Schauder Basis Which Is Not a Lipschitz Retract of Its Bidual},pdfauthor={Antonio Acuaviva}}
\usepackage{aliascnt}
\usepackage[nameinlink,noabbrev]{cleveref}

\theoremstyle{plain}
\newtheorem{theorem}{Theorem}[section]

\newaliascnt{lemma}{theorem}
\newtheorem{lemma}[lemma]{Lemma}
\aliascntresetthe{lemma}
\crefname{lemma}{lemma}{lemmas}
\Crefname{lemma}{Lemma}{Lemmas}

\newaliascnt{proposition}{theorem}
\newtheorem{proposition}[proposition]{Proposition}
\aliascntresetthe{proposition}
\crefname{proposition}{proposition}{propositions}
\Crefname{proposition}{Proposition}{Propositions}

\newaliascnt{corollary}{theorem}
\newtheorem{corollary}[corollary]{Corollary}
\aliascntresetthe{corollary}
\crefname{corollary}{corollary}{corollaries}
\Crefname{corollary}{Corollary}{Corollaries}

\theoremstyle{definition}
\newaliascnt{definition}{theorem}
\newtheorem{definition}[definition]{Definition}
\aliascntresetthe{definition}
\crefname{definition}{definition}{definitions}
\Crefname{definition}{Definition}{Definitions}

\newaliascnt{remark}{theorem}
\newtheorem{remark}[remark]{Remark}
\aliascntresetthe{remark}
\crefname{remark}{remark}{remarks}
\Crefname{remark}{Remark}{Remarks}

\newaliascnt{question}{theorem}

\aliascntresetthe{question}
\crefname{question}{question}{questions}
\Crefname{question}{Question}{Questions}

\newcommand{\N}{\mathbb N}
\newcommand{\R}{\mathbb R}
\newcommand{\C}{\mathbb C}
\newcommand{\K}{\mathbb K}

\newcommand{\eps}{\varepsilon}
\newcommand{\norm}[1]{ \| #1 \| }

\newcommand{\restr}[2]{\left.#1\right|_{#2}}

\DeclareMathOperator{\Lip}{Lip}

\DeclareMathOperator{\spn}{span}



\begin{document}

\title[Lipschitz retractions from biduals]{A Separable Banach Space with a Schauder Basis Which Is Not a Lipschitz Retract of Its Bidual}

\author[A. Acuaviva]{Antonio Acuaviva}
\address{School of Mathematical Sciences,
Fylde College,
Lancaster University,
LA1 4YF,
United Kingdom}
\email{ahacua@gmail.com}

\date{\today}

\begin{abstract}
    We construct a separable Banach space $X$ with a Schauder basis such that $B_X$ is not a uniformly continuous retract of $B_{X^{**}}$. Consequently, $X$ is not a uniformly continuous retract of $X^{**}$ and hence is not a Lipschitz retract of $X^{**}$.
\end{abstract}

\maketitle

\tableofcontents

\bigskip
\section{Introduction and organisation}\label{sec:introduction}

A longstanding problem, originating in Lindenstrauss's seminal 1964 paper on nonlinear projections, asks whether every Banach space $X$, canonically embedded in $X^{**}$, is a Lipschitz retract of $X^{**}$ \cite{Lindenstrauss1964}. Kalton gave a counterexample in the general case by constructing a nonseparable Lindenstrauss space $Z$ such that $B_Z$ is not a uniformly continuous retract of $B_{Z^{**}}$; in particular, $Z$ is not a Lipschitz retract of $Z^{**}$ \cite[Theorem~4.5]{Kalton2011}. Its restriction to separable Banach spaces, however, remained unresolved. Both the Lipschitz formulation and the weaker question of whether $B_X$ is a uniformly continuous retract of $B_{X^{**}}$ have subsequently been raised on several occasions; see, for example, \cite[Problems~4 and~5]{Kalton2012} and \cite[Problem~10]{GodefroyLancienZizler2014}.

This question may be viewed as a nonlinear counterpart of the classical problem of complementability in the bidual. It has an immediate affirmative answer when $X$ is reflexive or, more generally, when $X$ is linearly complemented in $X^{**}$. In his work, Lindenstrauss proved genuinely nonlinear positive examples, most notably $c_0$ and the spaces $C(K)$ for compact metrizable $K$, which are absolute Lipschitz retracts; see also \cite{Kalton2008}.

The problem is closely connected with two central themes in nonlinear Banach space theory. First, the Lipschitz free space construction linearises it: a Lipschitz retraction from $X^{**}$ onto $X$ yields a bounded projection from $\mathcal F(X^{**})$ onto $\mathcal F(X)$, and the converse follows by composing such a projection with the barycentre map \cite{GodefroyKalton2003,Kalton2008}. Secondly, the principle of local reflexivity places $X^{**}$ locally inside $X$ with arbitrarily small loss. Thus the problem asks whether this finite-dimensional linear information can always be assembled into a single global Lipschitz retraction. This relation between Lipschitz retractions, Lipschitz free spaces, and local complementation has been developed further in \cite{HajekQuilis2022}.

Our main result settles the separable case negatively, even within the class of spaces with a bimonotone Schauder basis.

\begin{theorem}\label{thm:main}
    There exists a separable Banach space $X$ with a bimonotone Schauder basis such that $B_X$ is not a uniformly continuous retract of $B_{X^{**}}$. In particular, $X$ is not a uniformly continuous retract of $X^{**}$ and hence is not a Lipschitz retract of $X^{**}$.
\end{theorem}

The basis assumption in \Cref{thm:main} is essentially optimal in the direction of unconditionality. Indeed, Kalton proved that every separable Banach space with an unconditional finite-dimensional decomposition is a Lipschitz retract of its bidual \cite[Theorem~5.2(i)]{Kalton2012}. Since every unconditional Schauder basis is a UFDD with one-dimensional components, the space in \Cref{thm:main} cannot be required to have an unconditional basis. Conversely, \Cref{prop:auxiliary-space} shows that the UFDD hypothesis in Kalton's theorem cannot be weakened to the existence of a monotone finite-dimensional decomposition. At the same time, the space in \Cref{thm:main} is Lipschitz-approximable with optimal common Lipschitz constant $1$; see \Cref{prop:soft-lipschitz-approximation}.

The nonretraction conclusion also extends to certain universal spaces. In particular, \Cref{cor:universal-basis} shows that the unit ball of Pe{\l}czy\'nski's complementably universal Banach space for the class of spaces with a Schauder basis is not a uniformly continuous retract of the unit ball of its bidual.

\subsection{Idea of the proof and organisation}
Kalton's nonseparable construction starts from the fact that the quotient map from $\ell_\infty$ onto $\ell_\infty/c_0$ has no uniformly continuous section on its unit ball and then amplifies this obstruction through equivalent renormings and a $c_0$-sum \cite{Kalton2011}. The present paper follows the quantitative core of that strategy, but replaces the nonseparable quotient by a separable summation quotient. For $m\in\N$, set $E_m=\spn\{e_1,\ldots,e_m\}\subseteq c_0$, and define
\begin{equation*}
    Q\colon Y=\left(\bigoplus_{m=1}^{\infty}E_m\right)_{\ell_1}\longrightarrow c_0, \qquad Q((u_m)_{m=1}^{\infty})=\sum_{m=1}^{\infty}u_m.
\end{equation*}
The map $Q$ admits an isometric linear lifting from $c_0$ into $Y^{**}$, while Kalton's obstruction for stable spaces implies that it has no uniformly continuous right inverse on $B_{c_0}$. For the equivalent norms
\begin{equation*}
    \norm{y}_n=\max\bigl\{2^{-n}\norm{y}_Y,\norm{Qy}_\infty\bigr\},
\end{equation*}
we prove the quantitative estimates for the bidual Lipschitz retraction constant
\begin{equation*}
    2^{n-1}\leq\Lambda(Y_n)\leq2^n.
\end{equation*}
Taking the $c_0$-sum
\begin{equation*}
    W=\left(\bigoplus_{n=1}^{\infty}Y_n\right)_{c_0}
\end{equation*}
produces an auxiliary separable space. A uniformly continuous retraction from $B_{W^{**}}$ onto $B_W$ would induce coordinate retractions with one common modulus of continuity. On a sufficiently high coordinate, comparison of the bidual lifting with genuine lifts in $Y_n$ then gives a uniformly continuous approximate right inverse for $Q$. The correction lemma upgrades it to an exact uniformly continuous right inverse on $B_{c_0}$, a contradiction. The estimates for $\Lambda(Y_n)$ record the corresponding quantitative obstruction at the Lipschitz level.

We also construct nested contractive finite-rank projections on $W$, and hence a monotone finite-dimensional decomposition. In particular, $W$ has the metric approximation property, and therefore the bounded approximation property. Pe{\l}czy\'nski's theorem then embeds $W$ as a complemented subspace of a Banach space with a Schauder basis. A transfer lemma shows that the unit-ball nonretraction property passes from a complemented subspace to the containing space and is invariant under isomorphism. An equivalent renorming makes the resulting basis monotone and completes the proof of \Cref{thm:main}.

The paper is organised as follows. In \Cref{sec:preliminaries}, we introduce the bidual Lipschitz retraction constant and establish the auxiliary norm and correction lemmas. In \Cref{sec:renorming}, we prove a general amplification theorem which, although not needed for the proof of the main uniform result, clarifies the quantitative mechanism underlying our choice of renormings. In \Cref{sec:quotient}, we analyse the summation quotient, construct a lifting into its bidual, and derive two-sided estimates for the associated renormed spaces. Finally, in \Cref{sec:counterexample}, we assemble these spaces into the auxiliary $c_0$-sum, prove the unit-ball obstruction, transfer it to spaces with Schauder bases, and record both the limitation of this transfer route and the Lipschitz approximability of the final space with optimal common constant $1$.

\bigskip
\section{Notation and preliminary results}\label{sec:preliminaries}

We use standard notation and conventions, unless explicitly stated other\-wise. All Banach spaces are over the field $\K\in\{\R,\C\}$. By an \emph{operator}, we always mean a bounded linear map. For a Banach space $X$, we denote by $B_X$ and $S_X$ its closed unit ball and unit sphere, respectively, and identify $X$ with its canonical image in $X^{**}$. We denote by $J_X$ the canonical inclusion of $X$ into its bidual. The identity operator on $X$ is denoted by $I_X$. 

A surjective operator $q\colon X\to Y$ is called a \emph{quotient map} if the norm of $Y$ is the quotient norm induced by $q$, that is,
\begin{equation*}
    \norm{y}=\inf\bigl\{\norm{x}:qx=y\bigr\}\qquad(y\in Y).
\end{equation*}
Equivalently, $q$ maps the open unit ball of $X$ onto the open unit ball of $Y$. This is the convention used in \cite{Kalton2004}.

\begin{definition}\label{def:retraction-constant}
Let $X$ be a Banach space. Its \emph{bidual Lipschitz retraction constant} is
\begin{equation*}
    \Lambda(X)=\inf\bigl\{\Lip(R):R\colon X^{**}\to X,\ \restr{R}{X}=I_X\bigr\},
\end{equation*}
where $\Lambda(X)=\infty$ if there is no Lipschitz retraction from $X^{**}$ onto $X$.
\end{definition}

Let us also recall a standard notion from the isometric theory of Banach spaces.

\begin{definition}
    A Banach space $X$ is called \emph{stable} if whenever $(x_n)_{n \in \N}$ and $(y_n)_{n \in \N}$ are two sequences in $X$, then, provided all limits exist,
    \begin{equation*}
    \lim_{n\to\infty}\lim_{m\to\infty}\|x_n+y_m\|=\lim_{m\to\infty}\lim_{n\to\infty}\|x_n+y_m\|.
    \end{equation*}
\end{definition}

We shall use the following form of an obstruction due to Kalton, which follows immediately from the proof of \cite[Theorem~7.3]{Kalton2004}.

\begin{theorem}[Kalton]\label{thm:kalton}
Let $Z$ be a stable real Banach space and let $q\colon Z\to c_0$ be an operator. Then there is no uniformly continuous map
\begin{equation*}
    \phi\colon B_{c_0}\longrightarrow Z
\end{equation*}
such that $q\phi=I_{B_{c_0}}$.
\end{theorem}

Indeed, any such map would be a uniform embedding of $B_{c_0}$ into $Z$, since
\begin{equation*}
    \norm{x-y}_{c_0}\leq\norm{q}\norm{\phi(x)-\phi(y)} \qquad(x,y\in B_{c_0}),
\end{equation*}
contrary to the obstruction used in Kalton's proof.

We shall also use the facts, recorded in the same paper, that the Johnson--Zippin space $C_1$ \cite{JohnsonZippin1972} admits an equivalent stable norm and that every $\ell_1$-sum of finite-dimensional spaces embeds isomorphically into $C_1$; see \cite[pp.~211--213]{Kalton2004}.

The first elementary ingredient identifies the bidual norm produced by the anisotropic renorming used below.

\begin{proposition}\label{prop:bidual-norm}
Let $Y, W$ be Banach spaces and $Q\colon Y\to W$ be an operator, let $\eps>0$, and equip $Y$ with the equivalent norm
\begin{equation*}
    \norm{y}_\eps=\max\bigl\{\eps\norm{y},\norm{Qy}\bigr\}.
\end{equation*}
Denote the resulting Banach space by $Y_\eps$. Under the canonical algebraic identification of $Y_\eps^{**}$ with $Y^{**}$, one has
\begin{equation*}
    \norm{z}_{\eps,**}=\max\bigl\{\eps\norm{z},\norm{Q^{**}z}\bigr\} \qquad(z\in Y^{**}).
\end{equation*}
\end{proposition}

\begin{proof}
Define
\begin{equation*}
    T_\eps\colon Y_\eps\longrightarrow Y\oplus_\infty W, \qquad T_\eps y=(\eps y,Qy).
\end{equation*}
The map $T_\eps$ is an isometry. Hence its second adjoint is an isometric embedding
\begin{equation*}
    T_\eps^{**}\colon Y_\eps^{**}\longrightarrow Y^{**}\oplus_\infty W^{**}.
\end{equation*}
For $z\in Y^{**}$, direct evaluation on $Y^*\oplus_1W^*$ gives
\begin{equation*}
    T_\eps^{**}z=(\eps z,Q^{**}z).
\end{equation*}
Taking the $\ell_\infty$-norm proves the formula.
\end{proof}

We shall also need a standard iterative correction of an approximate section. The proof is included because the quantitative value $1$ will be important later.

\begin{lemma}\label{lem:correction}
Let $E, W$ be Banach spaces and $q\colon E\to W$ be a quotient map. Suppose that a bounded uniformly continuous map
\begin{equation*}
    \phi\colon S_W\longrightarrow E
\end{equation*}
satisfies
\begin{equation*}
    \norm{q\phi(x)-x}\leq\lambda \qquad(x\in S_W)
\end{equation*}
for some $0\leq\lambda<1$. Then $q$ has a uniformly continuous right inverse on $B_W$.
\end{lemma}
\begin{proof}
If $W=\{0\}$, the conclusion is immediate, so assume that $W\neq\{0\}$. Extend $\phi$ positively homogeneously by defining
\begin{equation*}
    \Phi\colon B_W\longrightarrow E,\qquad \Phi(0)=0,\qquad \Phi(x)=\norm{x}\phi\left(\frac{x}{\norm{x}}\right)\quad(x\neq 0).
\end{equation*}
Since $\phi$ is bounded, the constant
\begin{equation*}
    C=\sup_{u\in S_W}\norm{\phi(u)}
\end{equation*}
is finite, and
\begin{equation*}
    \norm{\Phi(x)}\leq C\norm{x}\qquad(x\in B_W).
\end{equation*}
For $0\leq t\leq 2$, let
\begin{equation*}
    \omega(t)=\sup\bigl\{\norm{\phi(u)-\phi(v)}:u,v\in S_W,\ \norm{u-v}\leq t\bigr\}
\end{equation*}
be the modulus of uniform continuity of $\phi$. Then $\omega$ is nondecreasing and $\omega(t)$ tends to zero as $t$ tends to zero.

We first verify that $\Phi$ is uniformly continuous. It is enough to bound $\norm{\Phi(x)-\Phi(y)}$ by a quantity depending only on $\norm{x-y}$ and tending to zero as $\norm{x-y}$ tends to zero. Let $x,y\in B_W$, and, after interchanging $x$ and $y$ if necessary, set
\begin{equation*}
    r=\norm{x}\geq s=\norm{y},\qquad d=\norm{x-y}.
\end{equation*}
We may assume that $0<d\leq 1$. By the reverse triangle inequality,
\begin{equation*}
    0\leq r-s\leq d.
\end{equation*}

Suppose first that $s\leq\sqrt d$. Then
\begin{equation*}
    r+s=(r-s)+2s\leq d+2\sqrt d,
\end{equation*}
and therefore
\begin{equation*}
    \norm{\Phi(x)-\Phi(y)}\leq\norm{\Phi(x)}+\norm{\Phi(y)}\leq C(r+s)\leq C(d+2\sqrt d).
\end{equation*}

Suppose now that $s>\sqrt d$. Then $r\geq s>0$, so both $x$ and $y$ are nonzero. Moreover,
\begin{align*}
    \norm{\frac{x}{r}-\frac{y}{s}}
    &\leq \frac{\norm{x-y}}{r}+\norm{y}\left|\frac{1}{r}-\frac{1}{s}\right|\\
    &=\frac{d}{r}+\frac{r-s}{r}\\
    &\leq\frac{2d}{r}\leq\frac{2d}{s}<2\sqrt d.
\end{align*}
Consequently,
\begin{align*}
    \norm{\Phi(x)-\Phi(y)}
    &=\norm{r\phi\left(\frac{x}{r}\right)-s\phi\left(\frac{y}{s}\right)}\\
    &\leq (r-s)\norm{\phi\left(\frac{x}{r}\right)}+s\norm{\phi\left(\frac{x}{r}\right)-\phi\left(\frac{y}{s}\right)}\\
    &\leq C(r-s)+s\omega(2\sqrt d)\\
    &\leq Cd+\omega(2\sqrt d).
\end{align*}
In both cases, the resulting bound depends only on $d$ and tends to zero as $d$ tends to zero. Hence $\Phi$ is uniformly continuous on $B_W$.

The assumption on $\phi$ extends homogeneously from $S_W$ to $B_W$. Indeed, for every nonzero $x\in B_W$,
\begin{equation*}
    \norm{q\Phi(x)-x}=\norm{x}\norm{q\phi\left(\frac{x}{\norm{x}}\right)-\frac{x}{\norm{x}}}\leq\lambda\norm{x}.
\end{equation*}
The same inequality is immediate when $x=0$. Define
\begin{equation*}
    g\colon B_W\longrightarrow B_W,\qquad g(x)=x-q\Phi(x).
\end{equation*}
Since $\Phi$ is uniformly continuous and $q$ is bounded and linear, the map $g$ is uniformly continuous. Moreover,
\begin{equation*}
    \norm{g(x)}\leq\lambda\norm{x}\qquad(x\in B_W),
\end{equation*}
which also shows that $g(B_W)\subseteq B_W$.

Let $g^{\circ 0}$ denote the identity map on $B_W$, and let $g^{\circ k}$ denote the $k$-fold iterate of $g$ for $k\geq 1$. By an induction argument we get
\begin{equation*}
    \norm{g^{\circ k}(x)}\leq\lambda^k\norm{x}\qquad(x\in B_W,\ k\geq 0).
\end{equation*}
Define
\begin{equation*}
    s\colon B_W\longrightarrow E,\qquad s(x)=\sum_{k=0}^{\infty}\Phi\bigl(g^{\circ k}(x)\bigr).
\end{equation*}
For every $x\in B_W$ and $k\geq 0$,
\begin{equation*}
    \norm{\Phi\bigl(g^{\circ k}(x)\bigr)}\leq C\norm{g^{\circ k}(x)}\leq C\lambda^k.
\end{equation*}
Since $E$ is complete and the series $\sum_{k=0}^{\infty}C\lambda^k$ converges, the series defining $s$ converges uniformly on $B_W$.

For each $k\geq 0$, the map $\Phi\circ g^{\circ k}$ is uniformly continuous. Hence every partial sum of the series defining $s$ is uniformly continuous. Since these partial sums converge uniformly to $s$, the map $s$ is uniformly continuous.

Finally, the definition of $g$ gives
\begin{equation*}
    q\Phi(v)=v-g(v)\qquad(v\in B_W).
\end{equation*}
For $N\geq 0$, let
\begin{equation*}
    s_N(x)=\sum_{k=0}^{N}\Phi\bigl(g^{\circ k}(x)\bigr).
\end{equation*}
Then
\begin{equation*}
    qs_N(x) =\sum_{k=0}^{N}\bigl(g^{\circ k}(x)-g^{\circ(k+1)}(x)\bigr) =x-g^{\circ(N+1)}(x).
\end{equation*}
As $N$ tends to infinity, $s_N(x)$ converges to $s(x)$ and
\begin{equation*}
    \norm{g^{\circ(N+1)}(x)}\leq\lambda^{N+1}\norm{x}\longrightarrow 0.
\end{equation*}
Since $q$ is continuous, it follows that
\begin{equation*}
    qs(x)=x\qquad(x\in B_W).
\end{equation*}
Thus $s$ is a uniformly continuous right inverse of $q$ on $B_W$.
\end{proof}

\bigskip
\section{A quantitative renorming argument}\label{sec:renorming}

The following amplification theorem is not needed for the proof of our main uniform counterexample. Its purpose is to explain quantitatively the choice of renormings and to show that the bidual Lipschitz retraction constants of the coordinate spaces tend to infinity. This will also give an alternative proof of the Lipschitz obstruction for the auxiliary space used in the proof of \Cref{thm:main}.

\begin{theorem}\label{thm:amplification}
Let $Y, W$ be Banach spaces and let $Q\colon Y\to W$ be a quotient map. Suppose that there is an operator
\begin{equation*}
    S\colon W\longrightarrow Y^{**}
\end{equation*}
with
\begin{equation*}
    Q^{**}S=J_W \qquad\text{and}\qquad \norm{S}=M,
\end{equation*}
where $J_W\colon W\to W^{**}$ is the canonical embedding. Suppose also that $Q$ has no uniformly continuous right inverse on $B_W$.

Let $0<\eps<1$ satisfy $\eps M\leq1$, and equip $Y$ with the equivalent norm
\begin{equation*}
    \norm{y}_\eps=\max\bigl\{\eps\norm{y},\norm{Qy}\bigr\}.
\end{equation*}
Then
\begin{equation*}
    \Lambda(Y_\eps)\geq\frac{1}{\eps(M+1)}.
\end{equation*}
\end{theorem}

\begin{proof}
Suppose that
\begin{equation*}
    R\colon Y_\eps^{**}\longrightarrow Y_\eps
\end{equation*}
is an $L$-Lipschitz retraction. For $x\in S_W$, \Cref{prop:bidual-norm} and the assumption $\eps M\leq 1$ give
\begin{equation*}
    \norm{Sx}_{\eps,**}  =\max\bigl\{\eps\norm{Sx},\norm{Q^{**}Sx}\bigr\} =\max\bigl\{\eps\norm{Sx},\norm{x}\bigr\} =1.
\end{equation*}

Let $\delta>0$ be arbitrary subject to
\begin{equation*}
    \eps(1+\delta)\leq1.
\end{equation*}
Since $Q$ is a quotient map, for each $x\in S_W$ there is $y_x\in Y$ such that
\begin{equation*}
    Qy_x=x \qquad\text{and}\qquad \norm{y_x}<1+\delta.
\end{equation*}
It follows that $y_x\in B_{Y_\eps}$. Moreover,
\begin{equation*}
    Q^{**}(Sx-y_x)=0,
\end{equation*}
and hence another application of \Cref{prop:bidual-norm} gives
\begin{equation*}
    \norm{Sx-y_x}_{\eps,**} =\eps\norm{Sx-y_x} <\eps(M+1+\delta).
\end{equation*}
Since $R(y_x)=y_x$, we obtain
\begin{align*}
    \norm{QRSx-x}
    &=\norm{Q(RSx-y_x)}\leq\norm{RSx-y_x}_\eps\\
    &=\norm{RSx-Ry_x}_\eps<L\eps(M+1+\delta).
\end{align*}

The map $x\mapsto RSx$ is Lipschitz from $S_W$ into $Y_\eps$. Since the identity map from $Y_\eps$ to the original space $Y$ has norm at most $1/\eps$, the same map is bounded and uniformly continuous as a map from $S_W$ into $Y$.

If
\begin{equation*}
    L\eps(M+1+\delta)<1,
\end{equation*}
then \Cref{lem:correction} implies that $Q$ has a uniformly continuous right inverse on $B_W$, contrary to the hypothesis. Therefore,
\begin{equation*}
    L\eps(M+1+\delta)\geq 1
\end{equation*}
for every $\delta>0$ satisfying $\eps(1+\delta)\leq1$. Letting $\delta$ tend to zero, we obtain
\begin{equation*}
    L\eps(M+1)\geq 1.
\end{equation*}
Therefore,
\begin{equation*}
    L\geq\frac{1}{\eps(M+1)},
\end{equation*}
which proves the result.
\end{proof}

\begin{remark}\label{rem:no-regularity}
The choice $x\mapsto y_x$ in the proof of \Cref{thm:amplification} is not required to have any regularity. The points $y_x$ are used only as fixed points of the retraction against which the bidual lifting $Sx$ is compared. The resulting approximate section is the regular map $x\mapsto RSx$.
\end{remark}

\bigskip
\section{The summation quotient onto \texorpdfstring{$c_0$}{c0}}\label{sec:quotient}

For $m\in\N$, let
\begin{equation*}
    E_m=\spn\{e_1,\ldots,e_m\}\subseteq c_0.
\end{equation*}
Thus $E_m$ is isometric to $\ell_\infty^m$. Define
\begin{equation*}
    Y=\left(\bigoplus_{m=1}^{\infty}E_m\right)_{\ell_1}
\end{equation*}
and
\begin{equation*}
    Q\colon Y\longrightarrow c_0, \qquad Q((u_m)_{m=1}^{\infty})=\sum_{m=1}^{\infty}u_m.
\end{equation*}
The series converges absolutely in $c_0$. The relevant quotient and lifting properties of this construction are captured in the following proposition.

\begin{proposition}\label{prop:quotient-lifting}
The map $Q\colon Y\to c_0$ is a quotient map. Moreover, there is a linear isometry
\begin{equation*}
    S\colon c_0\longrightarrow Y^{**}
\end{equation*}
such that
\begin{equation*}
    Q^{**}S=J_{c_0}.
\end{equation*}
\end{proposition}
\begin{proof}
Let $P_m\colon c_0\to E_m$ denote the coordinate truncation. The estimate
\begin{equation*}
    \norm{\sum_{m=1}^{\infty}u_m}_\infty \leq\sum_{m=1}^{\infty}\norm{u_m}_\infty
\end{equation*}
shows that $\norm{Q}\leq1$.

Let $F\subseteq c_0$ be finite-dimensional and let $\eta>0$. Since $P_m\to I_{c_0}$ uniformly on $B_F$, choose $0<r<1$ and integers
\begin{equation*}
    m_1<m_2<\ldots
\end{equation*}
such that
\begin{equation*}
    \frac{r(1+r)}{1-r}<\eta \qquad\text{and}\qquad \norm{\restr{I_{c_0}-P_{m_k}}{F}}\leq r^k \quad(k\in\N).
\end{equation*}
For $x\in F$, define the local lifting $T_{F,\eta}x\in Y$ by
\begin{equation*}
    (T_{F,\eta}x)_{m_1}=P_{m_1}x,
\end{equation*}
\begin{equation*}
    (T_{F,\eta}x)_{m_k}=(P_{m_k}-P_{m_{k-1}})x \qquad(k\geq2),
\end{equation*}
and set all remaining coordinates equal to zero. The series telescopes, and therefore
\begin{equation*}
    QT_{F,\eta}x=x.
\end{equation*}
Furthermore,
\begin{equation*}
    \begin{aligned}
        \norm{T_{F,\eta}x}_Y
        &\leq\norm{x}
        +\sum_{k=2}^{\infty}
        \bigl(\norm{(I_{c_0}-P_{m_{k-1}})x}
        +\norm{(I_{c_0}-P_{m_k})x}\bigr)\\
        &\leq\left(1+\frac{r(1+r)}{1-r}\right)\norm{x}\\
        &\leq(1+\eta)\norm{x}.
    \end{aligned}
\end{equation*}
Applying this construction to $F=\spn\{x\}$ shows that $Q$ is onto and that the quotient norm induced by $Q$ is the usual norm of $c_0$. Hence $Q$ is a quotient map.

We next pass from the local liftings to a bidual lifting. Let
\begin{equation*}
    \mathcal D = \bigl\{ (F,\eta): F\subseteq c_0 \text{ is finite-dimensional and } 0<\eta\leq1 \bigr\}.
\end{equation*}
We equip $\mathcal D$ with the order defined by
\begin{equation*}
    (F,\eta)\preceq(G,\delta) \quad\Longleftrightarrow\quad F\subseteq G \text{ and } \delta\leq\eta.
\end{equation*}
Thus, moving forward in $\mathcal D$ amounts to enlarging the
finite-dimensional subspace and decreasing the error parameter. This makes
$\mathcal D$ a directed set, since any two elements $(F,\eta)$ and
$(G,\delta)$ have the common upper bound
\begin{equation*}
    \bigl(F+G,\min\{\eta,\delta\}\bigr).
\end{equation*}

For $i_0\in\mathcal D$, the tail beginning at $i_0$ is the set of all
indices lying beyond $i_0$, namely
\begin{equation*}
    \mathcal D_{i_0} = \{i\in\mathcal D:i_0\preceq i\}.
\end{equation*}
More explicitly, if $i_0=(F_0,\eta_0)$, then
\begin{equation*}
    \mathcal D_{i_0} = \bigl\{ (F,\eta)\in\mathcal D: F_0\subseteq F \text{ and } \eta\leq\eta_0 \bigr\}.
\end{equation*}

The family of tails forms a proper filter base on $\mathcal D$. Indeed, every tail is nonempty, and
\begin{equation*}
    \mathcal D_{(F,\eta)} \cap \mathcal D_{(G,\delta)} = \mathcal D_{\left(F+G,\min\{\eta,\delta\}\right)}.
\end{equation*}
Consequently, the tails generate a proper filter on $\mathcal D$, which, by the ultrafilter lemma, is contained in an ultrafilter $\mathcal U$. In particular, $\mathcal U$ contains every tail.
For $i=(F,\eta)\in\mathcal D$ and $x\in c_0$, put
\begin{equation*}
    v_i(x)=
    \begin{cases}
        T_{F,\eta}x,&x\in F,\\
        0,&x\notin F.
    \end{cases}
\end{equation*}
For each fixed $x$, the net $(v_i(x))_{i\in\mathcal D}$ is bounded. Define
\begin{equation*}
    Sx=w^*\!\operatorname{-lim}_{i\to\mathcal U}J_Yv_i(x).
\end{equation*}
The weak star limit exists in a bounded weak star compact ball of $Y^{**}$.

We claim that $S$ is linear. Fix $x,y\in c_0$ and scalars $a,b$. The set
\begin{equation*}
    A_{x,y}=\bigl\{(F,\eta)\in\mathcal D:x,y\in F\bigr\}
\end{equation*}
contains the tail beginning at
\begin{equation*}
    \bigl(\spn\{x,y\},1\bigr),
\end{equation*}
and hence belongs to $\mathcal U$. If $i=(F,\eta)\in A_{x,y}$, then
$ax+by\in F$, and therefore
\begin{equation*}
    v_i(ax+by)=T_{F,\eta}(ax+by)=aT_{F,\eta}x+bT_{F,\eta}y=av_i(x)+bv_i(y).
\end{equation*}
Thus this identity holds on a set belonging to $\mathcal U$. Applying
$J_Y$ and taking weak star limits along $\mathcal U$, we obtain
\begin{equation*}
    S(ax+by)=w^*\!\operatorname{-lim}_{i\to\mathcal U}J_Yv_i(ax+by).
\end{equation*}
Using the preceding identity on the set $A_{x,y}\in\mathcal U$, it follows
that
\begin{equation*}
    S(ax+by)=w^*\!\operatorname{-lim}_{i\to\mathcal U}\bigl(aJ_Yv_i(x)+bJ_Yv_i(y)\bigr).
\end{equation*}
Since addition and scalar multiplication are weak star continuous, we have
\begin{equation*}
    S(ax+by)=aSx+bSy.
\end{equation*}
Hence $S$ is linear.

We next estimate the norm of $S$. Fix $x\in c_0$ and $\varepsilon>0$. The
set
\begin{equation*}
    A_{x,\varepsilon}=\bigl\{(F,\eta)\in\mathcal D:x\in F\text{ and }\eta<\varepsilon\bigr\}
\end{equation*}
belongs to $\mathcal U$. Indeed, the condition $x\in F$ holds on a tail,
while $\eta\to0$ along $\mathcal U$. For every $i=(F,\eta)$ in this set,
\begin{equation*}
    \norm{v_i(x)}=\norm{T_{F,\eta}x}\leq(1+\eta)\norm{x}\leq(1+\varepsilon)\norm{x}.
\end{equation*}
Since the ball
\begin{equation*}
    (1+\varepsilon)\norm{x}B_{Y^{**}}
\end{equation*}
is weak star closed, it contains the weak star ultralimit $Sx$. Hence
\begin{equation*}
    \norm{Sx}\leq(1+\varepsilon)\norm{x}.
\end{equation*}
As $\varepsilon>0$ was arbitrary,
\begin{equation*}
    \norm{Sx}\leq\norm{x}.
\end{equation*}

Finally, for fixed $x\in c_0$, the set
\begin{equation*}
    A_x=\bigl\{(F,\eta)\in\mathcal D:x\in F\bigr\}
\end{equation*}
belongs to $\mathcal U$. On this set,
\begin{equation*}
    Qv_i(x)=QT_{F,\eta}x=x.
\end{equation*}
Using the identity
\begin{equation*}
    Q^{**}J_Y=J_{c_0}Q
\end{equation*}
and the weak star continuity of $Q^{**}$, we obtain
\begin{equation*}
    Q^{**}Sx=w^*\!\operatorname{-lim}_{i\to\mathcal U}Q^{**}J_Yv_i(x).
\end{equation*}
Consequently,
\begin{equation*}
    Q^{**}Sx=w^*\!\operatorname{-lim}_{i\to\mathcal U}J_{c_0}Qv_i(x)=J_{c_0}x.
\end{equation*}
Since $J_{c_0}$ is an isometry and $Q^{**}$ is contractive,
\begin{equation*}
    \norm{x}=\norm{J_{c_0}x}=\norm{Q^{**}Sx}\leq\norm{Sx}.
\end{equation*}
Combining this with the reverse inequality gives
\begin{equation*}
    \norm{Sx}=\norm{x}
\end{equation*}
for every $x\in c_0$, so $S$ is an isometry.
\end{proof}

Although $Q^{**}$ admits an isometric lifting of the canonical copy of
$c_0$, the quotient map $Q$ itself has no uniformly continuous section on the unit ball, as we show now.

\begin{proposition}\label{prop:no-section}
The quotient $Q\colon Y\to c_0$ has no uniformly continuous right inverse on $B_{c_0}$.
\end{proposition}

\begin{proof}
Assume first that the scalar field is $\R$. Kalton records that every $\ell_1$-sum of finite-dimensional spaces embeds isomorphically into the Johnson--Zippin space $C_1$, and that $C_1$ admits an equivalent stable norm; see \cite[pp.~211--213]{Kalton2004}. Stability passes to subspaces, so pulling the restricted stable norm back to $Y$ gives an equivalent stable norm on the same vector space.

Suppose that $Q$ had a uniformly continuous right inverse on $B_{c_0}$ for the original norm of $Y$. Equivalent norms are globally bi-Lipschitz, so the same right inverse would remain uniformly continuous for the stable equivalent norm. The operator $Q$ remains bounded under equivalent renorming, contradicting \Cref{thm:kalton}.

Now suppose that the scalar field is $\C$. If a uniformly continuous right inverse
\begin{equation*}
    \psi\colon B_{c_0(\C)}\longrightarrow \left(\bigoplus_{m=1}^{\infty}\ell_\infty^m(\C)\right)_{\ell_1}
\end{equation*}
existed, its restriction to $B_{c_0(\R)}$, followed by coordinatewise real part, would give a uniformly continuous right inverse for the corresponding real summation quotient. This is impossible by the preceding paragraph.
\end{proof}

For $n\in\N$, equip the vector space $Y$ with the norm
\begin{equation*}
    \norm{y}_n=\max\bigl\{2^{-n}\norm{y}_Y,\norm{Qy}_\infty\bigr\},
\end{equation*}
and denote the resulting Banach space by $Y_n$. We will carry this notation for the rest of the paper. We record the following quantitative statement.

\begin{proposition}\label{prop:two-sided}
For every $n\in\N$,
\begin{equation*}
    2^{n-1}\leq\Lambda(Y_n)\leq2^n.
\end{equation*}
\end{proposition}

\begin{proof}
By \Cref{prop:quotient-lifting,prop:no-section}, the assumptions of \Cref{thm:amplification} hold with $W=c_0$, $M=1$, and $\eps=2^{-n}$. Hence
\begin{equation*}
    \Lambda(Y_n)\geq\frac{1}{2\cdot2^{-n}}=2^{n-1}.
\end{equation*}

For the upper estimate, put
\begin{equation*}
    Z=\left(\bigoplus_{m=1}^{\infty}E_m^*\right)_{c_0}.
\end{equation*}
Since each $E_m$ is finite-dimensional, $Y=Z^*$ isometrically. Consequently, the canonical map
\begin{equation*}
    P=J_Z^*\colon Y^{**}=Z^{***}\longrightarrow Z^*=Y
\end{equation*}
is a norm-one projection onto the canonical copy of $Y$. Under the algebraic identification of $Y_n^{**}$ with $Y^{**}$, \Cref{prop:bidual-norm} gives
\begin{equation*}
    \norm{z}_{Y^{**}}\leq2^n\norm{z}_{n,**} \qquad(z\in Y^{**}).
\end{equation*}
Moreover,
\begin{equation*}
        \norm{Pz}_n =\max\bigl\{2^{-n}\norm{Pz}_Y,\norm{QPz}_\infty\bigr\}\leq\norm{Pz}_Y\leq\norm{z}_{Y^{**}}\leq2^n\norm{z}_{n,**}.
\end{equation*}
Thus $P\colon Y_n^{**}\to Y_n$ is a linear retraction of norm at most $2^n$, which proves the upper estimate.
\end{proof}

The estimates in \Cref{prop:two-sided} are not needed to rule out a uniformly continuous retraction between the unit balls of the auxiliary space in \Cref{sec:counterexample}. They instead provide a quantitative Lipschitz counterpart to that argument and yield an alternative proof that this auxiliary space is not a Lipschitz retract of its bidual; see \Cref{rem:mechanism}.

\bigskip
\section{The uniform counterexample}\label{sec:counterexample}

We first combine the renormed spaces from \Cref{sec:quotient} in one auxiliary separable Banach space. Put
\begin{equation*}
    W=\left(\bigoplus_{n=1}^{\infty}Y_n\right)_{c_0}.
\end{equation*}
The space $W$ is separable because every $Y_n$ is separable, and its bidual is canonically isometric to
\begin{equation*}
    W^{**}=\left(\bigoplus_{n=1}^{\infty}Y_n^{**}\right)_{\ell_\infty}.
\end{equation*}

\begin{proposition}\label{prop:auxiliary-space}
    The space $W$ has a monotone finite-dimensional decomposition, and $B_W$ is not a uniformly continuous retract of $B_{W^{**}}$.
\end{proposition}

\begin{proof}
Let
\begin{equation*}
    P_N\colon c_0\longrightarrow E_N
\end{equation*}
be the coordinate truncation. For $N\in\N$, define a finite-rank operator
\begin{equation*}
    A_N\colon Y\longrightarrow Y
\end{equation*}
as follows. If $y=(u_m)_{m=1}^{\infty}\in Y$, put
\begin{equation*}
    (A_Ny)_m=
    \begin{cases}
        u_m,&m<N,\\
        P_N\left(\displaystyle\sum_{k=N}^{\infty}u_k\right),&m=N,\\
        0,&m>N.
    \end{cases}
\end{equation*}
Since $u_m\in E_m\subseteq E_N$ for $m<N$, we have
\begin{equation*}
    QA_Ny=P_NQy.
\end{equation*}
Moreover,
\begin{equation*}
    \norm{A_Ny}_Y \leq \sum_{m=1}^{N-1}\norm{u_m}_\infty + \sum_{m=N}^{\infty}\norm{u_m}_\infty = \norm{y}_Y.
\end{equation*}
Consequently, for every $n\in\N$,
\begin{align*}
    \norm{A_Ny}_n
    &=\max\bigl\{2^{-n}\norm{A_Ny}_Y,\norm{P_NQy}_\infty\bigr\}\\
    &\leq\max\bigl\{2^{-n}\norm{y}_Y,\norm{Qy}_\infty\bigr\}=\norm{y}_n.
\end{align*}
We also record the compatibility of these operators. Suppose, for example,
that $M<N$. Since $A_N$ fixes every vector in the range of $A_M$, we have
$A_NA_M=A_M$. For the reverse composition, the only nontrivial coordinate is
the $M$th one, where
\begin{equation*}
    P_M\left(\sum_{k=M}^{N-1}u_k+P_N\sum_{k=N}^{\infty}u_k\right)=P_M\sum_{k=M}^{\infty}u_k,
\end{equation*}
because $P_MP_N=P_M$. Hence $A_MA_N=A_M$, and therefore
\begin{equation*}
    A_NA_M=A_MA_N=A_{\min\{N,M\}} \qquad(M,N\in\N).
\end{equation*}
Furthermore,
\begin{equation*}
    \norm{A_Ny-y}_Y \leq 2\sum_{m>N}\norm{u_m}_\infty, 
\end{equation*}
and hence
\begin{equation*}
    A_Ny\longrightarrow y \quad \text{ as } \quad N \longrightarrow \infty
\end{equation*}
in every space $Y_n$.

For
\begin{equation*}
    w=(y_n)_{n=1}^{\infty}\in W,
\end{equation*}
define
\begin{equation*}
    \Pi_Nw=(A_Ny_1,\ldots,A_Ny_N,0,0,\ldots).
\end{equation*}
Each $\Pi_N$ is a contractive finite-rank projection, and
\begin{equation*}
    \Pi_N\Pi_M=\Pi_M\Pi_N=\Pi_{\min\{N,M\}} \qquad(M,N\in\N).
\end{equation*}
We claim that, for every fixed $w\in W$,
\begin{equation*}
    \Pi_Nw\longrightarrow w  \quad \text{ as } \quad N \longrightarrow \infty.
\end{equation*}
Indeed, let $w=(y_n)_{n=1}^{\infty}\in W$ and let $\eps>0$. Choose $M\in\N$ such that
\begin{equation*}
    \sup_{n>M}\norm{y_n}_n<\eps.
\end{equation*}
For $N\geq M$, contractivity of $A_N$ gives
\begin{equation*}
    \norm{\Pi_Nw-w}_W \leq \max\left\{ \max_{1\leq n\leq M}\norm{A_Ny_n-y_n}_n, 2\sup_{n>M}\norm{y_n}_n \right\}.
\end{equation*}
The first term tends to zero as $N$ tends to infinity, while the second is at most $2\eps$. This proves the claim.

Set
\begin{equation*}
    \Pi_0=0, \qquad D_N=\Pi_N-\Pi_{N-1}, \qquad F_N=D_NW \quad(N\in\N).
\end{equation*}
The compatibility of the projections gives
\begin{equation*}
    D_ND_M=0\quad(N\neq M), \qquad \sum_{k=1}^{N}D_k=\Pi_N.
\end{equation*}
Then every $F_N$ is finite-dimensional,
\begin{equation*}
    W=\overline{\bigoplus_{N=1}^{\infty}F_N},
\end{equation*}
and the associated partial-sum projections are precisely the operators $\Pi_N$. Thus $(F_N)_{N=1}^{\infty}$ is a monotone finite-dimensional decomposition of $W$.

We now prove the retraction assertion. The argument exploits the fact that a single uniformly continuous map on the bidual unit ball has one modulus of continuity controlling every coordinate. Proceed by contradiction and suppose that
\begin{equation*}
    R\colon B_{W^{**}}\longrightarrow B_W
\end{equation*}
is a uniformly continuous retraction. Choose $\eta>0$ such that
\begin{equation*}
    \norm{Rz-Rz'}_W<\frac{1}{2}
\end{equation*}
whenever $z,z'\in B_{W^{**}}$ satisfy
\begin{equation*}
    \norm{z-z'}_{W^{**}}<\eta,
\end{equation*}
and then choose $n\in\N$ so large that
\begin{equation*}
    3\cdot2^{-n}<\eta.
\end{equation*}

For each $k\in\N$, let
\begin{equation*}
    j_k\colon Y_k^{**}\longrightarrow W^{**}
\end{equation*}
be the $k$th coordinate embedding, and let
\begin{equation*}
    \pi_k\colon W\longrightarrow Y_k
\end{equation*}
be the $k$th coordinate projection. Both maps have norm one. For the index $n$ chosen above, put
\begin{equation*}
    r_n=\pi_nRj_n\colon B_{Y_n^{**}}\longrightarrow B_{Y_n}.
\end{equation*}
Then $r_n$ is a uniformly continuous retraction. Since $j_n$ and $\pi_n$ have norm one, it also satisfies
\begin{equation*}
    \norm{r_nz-r_nz'}_n<\frac{1}{2}
\end{equation*}
whenever $z,z'\in B_{Y_n^{**}}$ satisfy
\begin{equation*}
    \norm{z-z'}_{n,**}<\eta.
\end{equation*}

Let $Q\colon Y\to c_0$ and $S\colon c_0\to Y^{**}$ be the maps from \Cref{prop:quotient-lifting}. By \Cref{prop:bidual-norm},
\begin{equation*}
    \norm{Sx}_{n,**} =\max\bigl\{2^{-n}\norm{Sx},\norm{x}_\infty\bigr\} =\norm{x}_\infty \qquad(x\in c_0).
\end{equation*}
For $x\in S_{c_0}$, choose $y_x\in Y$ such that
\begin{equation*}
    Qy_x=x \qquad\text{and}\qquad \norm{y_x}_Y<2.
\end{equation*}
Then
\begin{equation*}
    \norm{y_x}_n =\max\bigl\{2^{-n}\norm{y_x}_Y,\norm{x}_\infty\bigr\} =1,
\end{equation*}
so $y_x\in B_{Y_n}$. Moreover,
\begin{equation*}
    Q^{**}(Sx-y_x)=0.
\end{equation*}
Hence
\begin{equation*}
    \norm{Sx-y_x}_{n,**} =2^{-n}\norm{Sx-y_x} <3\cdot2^{-n} <\eta.
\end{equation*}
Since $r_n(y_x)=y_x$, it follows that
\begin{equation*}
    \norm{r_n(Sx)-y_x}_n<\frac{1}{2}.
\end{equation*}

Define
\begin{equation*}
    \phi\colon S_{c_0}\longrightarrow Y, \qquad \phi(x)=r_n(Sx).
\end{equation*}
The map $\phi$ is bounded and uniformly continuous. Indeed, $S\colon c_0\to Y_n^{**}$ is an isometry, $r_n$ is uniformly continuous, and the identity map from $Y_n$ to $Y$ has norm at most $2^n$. Moreover,
\begin{equation*}
    \norm{Q\phi(x)-x}_\infty=\norm{Q(r_n(Sx)-y_x)}_\infty\leq\norm{r_n(Sx)-y_x}_n<\frac{1}{2}\qquad(x\in S_{c_0}).
\end{equation*}
By \Cref{lem:correction}, the quotient map $Q$ would have a uniformly continuous right inverse on $B_{c_0}$, contrary to \Cref{prop:no-section}. This contradiction proves that no uniformly continuous retraction from $B_{W^{**}}$ onto $B_W$ exists.
\end{proof}

\begin{remark}\label{rem:mechanism}
The nonretraction part of the proof above is driven by a common modulus of continuity across all coordinates. At the Lipschitz level, \Cref{prop:two-sided} gives the same mechanism in quantitative form. Indeed, if
\begin{equation*}
    T\colon W^{**}\longrightarrow W
\end{equation*}
were an $L$-Lipschitz retraction, then
\begin{equation*}
    \pi_nTj_n\colon Y_n^{**}\longrightarrow Y_n
\end{equation*}
would be an $L$-Lipschitz retraction for every $n$. Consequently,
\begin{equation*}
    L\geq\Lambda(Y_n)\geq2^{n-1} \qquad(n\in\N),
\end{equation*}
which is impossible. Thus the unbounded constants $\Lambda(Y_n)$ provide a direct alternative proof that $W$ is not a Lipschitz retract of $W^{**}$, while \Cref{prop:auxiliary-space} rules out the more general uniformly continuous retractions between the unit balls.
\end{remark}

The following elementary transfer principle will allow us to pass from the auxiliary space $W$ to a space with a Schauder basis.

\begin{lemma}\label{lem:complemented-transfer}
Let $E$ and $Z$ be Banach spaces, and suppose that there are operators
\begin{equation*}
    U\colon E\longrightarrow Z \qquad\text{and}\qquad V\colon Z\longrightarrow E
\end{equation*}
such that
\begin{equation*}
    VU=I_E.
\end{equation*}
If $B_Z$ is a uniformly continuous retract of $B_{Z^{**}}$, then $B_E$ is a uniformly continuous retract of $B_{E^{**}}$. Consequently, failure of this retraction property passes from a complemented subspace to the containing space and is invariant under isomorphism.
\end{lemma}

\begin{proof}
Suppose that
\begin{equation*}
    R\colon B_{Z^{**}}\longrightarrow B_Z
\end{equation*}
is a uniformly continuous retraction, and put
\begin{equation*}
    a=\max\{1,\norm{U}\}.
\end{equation*}
Let
\begin{equation*}
    \rho_E\colon E\longrightarrow B_E, \qquad \rho_E(x)=\frac{x}{\max\{1,\norm{x}\}},
\end{equation*}
be the radial retraction. Define
\begin{equation*}
    \widetilde R\colon B_{E^{**}}\longrightarrow B_E
\end{equation*}
by
\begin{equation*}
    \widetilde R(x^{**}) = \rho_E\left( aV R\left(a^{-1}U^{**}x^{**}\right) \right).
\end{equation*}
Since
\begin{equation*}
    a^{-1}U^{**}(B_{E^{**}})\subseteq B_{Z^{**}},
\end{equation*}
the map $\widetilde R$ is well defined and uniformly continuous.

If $x\in B_E$, then $a^{-1}Ux\in B_Z$. Under the canonical embeddings of
$E$ and $Z$ into their biduals, one has $U^{**}x=Ux$, and hence
\begin{equation*}
    \widetilde R(x)=\rho_E\left(aV R\left(a^{-1}Ux\right)\right)=\rho_E(VUx)=\rho_E(x)=x.
\end{equation*}
Thus $\widetilde R$ is a uniformly continuous retraction from $B_{E^{**}}$ onto $B_E$.

If $U\colon E\to Z$ is an isomorphism, we may take $V=U^{-1}$ and apply the preceding implication in both directions. This proves the asserted isomorphic invariance.
\end{proof}

\begin{proof}[Proof of \Cref{thm:main}]
By \Cref{prop:auxiliary-space}, the space $W$ has a monotone finite-dimensional decomposition. In particular, $W$ has the metric approximation property, and hence the bounded approximation property. Therefore, by Pe{\l}czy\'nski's theorem \cite[Theorem~1]{Pelczynski1971}, there is a Banach space $Z$ with a Schauder basis for which there are operators
\begin{equation*}
    U\colon W\longrightarrow Z \qquad\text{and}\qquad V\colon Z\longrightarrow W
\end{equation*}
satisfying
\begin{equation*}
    VU=I_W.
\end{equation*}
Since $B_W$ is not a uniformly continuous retract of $B_{W^{**}}$, \Cref{lem:complemented-transfer} shows that $B_Z$ is not a uniformly continuous retract of $B_{Z^{**}}$. Every Schauder basis admits an equivalent norm under which it is bimonotone. Equip $Z$ with such a norm and denote the resulting Banach space by $X$. Then $X$ is separable and has a bimonotone Schauder basis. Since the identity between $Z$ and $X$ is an isomorphism, the isomorphic invariance in \Cref{lem:complemented-transfer} shows that $B_X$ is not a uniformly continuous retract of $B_{X^{**}}$.

Finally, if $T\colon X^{**}\to X$ were a uniformly continuous retraction, then composing its restriction to $B_{X^{**}}$ with the $2$-Lipschitz radial retraction $\rho_X(x)=x/\max\{1,\norm{x}\}$ would give a uniformly continuous retraction from $B_{X^{**}}$ onto $B_X$, a contradiction. Therefore $X$ is not a uniformly continuous, and hence not a Lipschitz, retract of $X^{**}$.
\end{proof}

Let $\mathbb U_{\mathrm B}$ denote Pe{\l}czy\'nski's complementably universal Banach space for the class of Banach spaces with a Schauder basis \cite{Pelczynski1969}. As an automatic consequence, we get the following.

\begin{corollary}\label{cor:universal-basis}
    The unit ball $B_{\mathbb U_{\mathrm B}}$ is not a uniformly continuous retract of $B_{\mathbb U_{\mathrm B}^{**}}$. Consequently, $\mathbb U_{\mathrm B}$ is not a uniformly continuous, and hence not a Lipschitz, retract of $\mathbb U_{\mathrm B}^{**}$.
\end{corollary}

\begin{proof}
By the complementable universality of $\mathbb U_{\mathrm B}$ \cite[Corollary~1]{Pelczynski1969}, the space $X$ from \Cref{thm:main} is isomorphic to a complemented subspace of $\mathbb U_{\mathrm B}$. The ball assertion follows from \Cref{lem:complemented-transfer}, and the whole-space assertion follows from the radial-retraction argument used in \Cref{thm:main}.
\end{proof}

\begin{remark}\label{rem:no-shrinking-basis}
This method cannot produce a counterexample with a shrinking basis. Indeed, $Y^*=(\bigoplus_{m=1}^{\infty}E_m^*)_{\ell_\infty}$ contains an isometric copy of $\ell_\infty$. Since each $Y_n$ is an equivalent renorming of $Y$, the space $W^*=(\bigoplus_{n=1}^{\infty}Y_n^*)_{\ell_1}$ is nonseparable. As $W$ is isomorphic to a complemented subspace of both $X$ and $\mathbb U_{\mathrm B}$, their duals are also nonseparable. A shrinking Schauder basis forces the dual to be separable, and dual separability is invariant under equivalent renorming, so obtaining a counterexample with a shrinking basis, or more generally with separable dual, would require a different construction.
\end{remark}

We finish with an observation concerning Lipschitz approximability in the sense of Godefroy \cite[Definition~1.1]{Godefroy2020}.

\begin{proposition}\label{prop:soft-lipschitz-approximation}
Let $E$ be a nonzero separable Banach space.
\begin{enumerate}
\renewcommand{\labelenumi}{(\roman{enumi})}
\renewcommand{\theenumi}{(\roman{enumi})}

\item\label{item:soft-lipschitz-extension}
Suppose that, for some $\lambda\geq1$, there are compact operators $T_N\colon E\to E$ such that $\norm{T_N}\leq\lambda$ and $T_Nx\to x$ for every $x\in E$. Then there are maps $\Psi_N\colon E^{**}\to E$ with relatively compact ranges such that $\Lip(\Psi_N)\leq\lambda$ and $\Psi_N(x)\to x$ for every $x\in E$. In particular, $E$ is Lipschitz-approximable with common constant $\lambda$.

\item\label{item:soft-lipschitz-mcap}
If $E$ has the metric compact approximation property, then it is Lipschitz-approximable with optimal common constant $1$. Consequently, the same holds whenever $E$ has the metric approximation property.
\end{enumerate}
\end{proposition}

\begin{proof}
We first prove part~\ref{item:soft-lipschitz-extension}. Let $J_E\colon E\to E^{**}$ denote the canonical embedding. Since each $T_N$ is compact, $T_N^{**}$ is compact and has range contained in $J_E(E)$. Define
\begin{equation*}
    \widehat T_N=J_E^{-1}\circ T_N^{**}\colon E^{**}\longrightarrow E.
\end{equation*}
Then $\widehat T_N$ is compact, $\norm{\widehat T_N}=\norm{T_N}$, and $\widehat T_NJ_E=T_N$. Set
\begin{equation*}
    H_N(z)=\frac{N}{N+\norm{z}}\,z,\qquad \Psi_N=\widehat T_N\circ H_N.
\end{equation*}
The map $H_N$ is nonexpansive. Indeed, if $r=\norm{z}\geq s=\norm{w}$, $d=\norm{z-w}$, $\alpha=N/(N+r)$, and $\beta=N/(N+s)$, then $r-s\leq d$ and
\begin{equation*}
    \norm{H_N(z)-H_N(w)}\leq\alpha d+(\beta-\alpha)s\leq\frac{N(N+2s)}{(N+r)(N+s)}\,d\leq d.
\end{equation*}
The last inequality follows from $(N+r)(N+s)-N(N+2s)=N(r-s)+rs\geq0$. Hence $\Lip(\Psi_N)\leq\lambda$, and the compactness of $\widehat T_N$, together with $H_N(E^{**})\subseteq NB_{E^{**}}$, shows that $\Psi_N(E^{**})$ is relatively compact.

For $x\in E$, identified with its canonical image in $E^{**}$, one has $\Psi_N(x)=T_N(Nx/(N+\norm{x}))$ and
\begin{equation*}
    \norm{\Psi_N(x)-x}\leq\lambda\frac{\norm{x}^2}{N+\norm{x}}+\norm{T_Nx-x}\longrightarrow0.
\end{equation*}
Thus the restrictions of $\Psi_N$ to $E$ form the required Lipschitz approximating sequence. This proves part~\ref{item:soft-lipschitz-extension}.

We now prove part~\ref{item:soft-lipschitz-mcap}. Suppose that $E$ has the metric compact approximation property. Choose a dense sequence $(x_j)_{j=1}^{\infty}$ and compact contractions $K_N\colon E\to E$ such that
\begin{equation*}
    \max_{1\leq j\leq N}\norm{K_Nx_j-x_j}<\frac{1}{N}.
\end{equation*}
For $N\geq j$, one has $\norm{K_Nx-x}\leq2\norm{x-x_j}+1/N$, so density gives $K_Nx\to x$ for every $x\in E$. Part~\ref{item:soft-lipschitz-extension} therefore applies with $\lambda=1$. The metric approximation property is the finite-rank special case.

Finally, if $(\varphi_k)_{k \in \N}$ is any Lipschitz approximating sequence with common constant $C$, then for distinct $x,y\in E$ pointwise convergence gives
\begin{equation*}
    \norm{x-y}=\lim_{k\to\infty}\norm{\varphi_k(x)-\varphi_k(y)}\leq C\norm{x-y}.
\end{equation*}
Hence $C\geq1$, proving optimality and completing the proof of part~\ref{item:soft-lipschitz-mcap}.
\end{proof}

\begin{remark}\label{rem:lipschitz-approximable}
Applying \Cref{prop:soft-lipschitz-approximation} to the partial-sum projections shows that the space $X$ in \Cref{thm:main} is Lipschitz-approximable with optimal common Lipschitz constant $1$. The maps $\Psi_N$ do not give retractions: their ranges are relatively compact, and they converge to the identity only on the canonical copy of $X$. Thus the construction does not answer Godefroy's question negatively; rather, it shows that optimal Lipschitz approximability, even together with the metric approximation property and a monotone Schauder basis, does not imply uniformly continuous or Lipschitz retractability from the bidual.
\end{remark}

\bigskip

\noindent\textbf{Acknowledgements.} This paper forms part of the author's PhD research at Lancaster University, conducted under the supervision of Professor N. J. Laustsen. He acknowledges with thanks the funding from the EPSRC (grant number EP/W524438/1) that has supported his studies. \medskip

\noindent\textbf{Statement on AI use.} Large language models, in particular OpenAI's ChatGPT 5.6 Pro, were used during the exploratory and preparatory stages of this work. The author proposed and directed the renorming amplification strategy underlying the construction. ChatGPT assisted in working out parts of the technical implementation, including auxiliary lemmas and technical details, as well as with literature retrieval, consistency checks, and \LaTeX{} preparation. The author takes full responsibility for the mathematical content of the paper. \medskip

For the purpose of open access, the author has applied a Creative Commons Attribution (CC BY) licence to any Author Accepted Manuscript version arising. \medskip

\noindent\textbf{Data availability.} No data was used for the research described in the article.

\end{document}